\documentclass[11pt]{article}
\usepackage{pifont}
\usepackage{}
\usepackage{amsfonts}
\usepackage{curves}
\usepackage{latexsym}
\usepackage{color}
\usepackage{amsmath, epsfig}

\usepackage{mathrsfs}
\usepackage{amssymb}


 
\newtheorem{thm}{Theorem}[section]
\newcommand{\n}{\noindent}

\newtheorem{lemma}[thm]{Lemma}

\newtheorem{prob}[thm]{Problem}

\newenvironment{proof}{{\bf Proof}.}{\rule{3mm}{3mm}}
\begin{document}

\title{Realizing degree sequences as $Z_3$-connected graphs}
\author{Fan Yang$^1$~\thanks{Supported by the Natural Science Foundation of China (11326215)}, Xiangwen Li$^2$~\thanks{Supported by the
Natural Science Foundation of China (11171129)  and by Doctoral Fund of Ministry of Education of China (20130144110001)}\thanks{Corresponding author: xwli68@mail.ccnu.edu.cn}, Hong -Jian Lai$^3$\\
$^1$Department of Mathematics and Physics, Jiangsu University of Science and Technology, \\Zhenjiang 212003,
 China\\
$^2$Department of Mathematics, Huazhong Normal
University, Wuhan 430079, China\\
$^3$Department of Mathematics, West Virginia University,
Morgantown, WV 26506, USA}

\date{}
\maketitle

\begin{abstract}
An integer-valued sequence $\pi=(d_1, \ldots, d_n)$ is  {\em graphic} if
there is a simple graph $G$ with degree sequence of $\pi$. We say the $\pi$ has a realization $G$. Let $Z_3$ be a cyclic group of order three. A graph $G$ is {\em $Z_3$-connected} if for every mapping $b:V(G)\to Z_3$ such that $\sum_{v\in V(G)}b(v)=0$, there is an orientation of $G$  and a mapping $f: E(G)\to Z_3-\{0\}$ such that for each vertex $v\in V(G)$, the sum of  the values of $f$ on all the edges leaving from $v$ minus the sum of the values of $f$ on the all edges coming to $v$ is equal to $b(v)$.  If  an integer-valued sequence $\pi$ has a realization $G$ which is $Z_3$-connected, then $\pi$ has a {\em $Z_3$-connected realization} $G$.
Let $\pi=(d_1, \ldots, d_n)$ be a graphic sequence with $d_1\ge \ldots \ge d_n\ge 3$.
We prove in this paper that if $d_1\ge n-3$, then either $\pi$ has a $Z_3$-connected realization
unless the sequence is $(n-3, 3^{n-1})$ or is $(k, 3^k)$ or $(k^2, 3^{k-1})$ where $k=n-1$ and $n$ is even;  if $d_{n-5}\ge 4$, then
either $\pi$ has a $Z_3$-connected realization unless the sequence is
 $(5^2, 3^4)$ or $(5, 3^5)$.
\end{abstract}

\section{Introduction}

Graphs here are finite, and may have multiple edges without
loops. We follow the notation and terminology in~\cite{BM76} except otherwise stated.

For a given orientation of a graph  $G$, if an edge $e\in E(G)$ is
directed from a vertex $u$ to a vertex $v$, then $u$ is the {\em tail} of
$e$ and $v$ is the {\em head} of $e$. For a vertex $v\in V(G)$, let
$E^+(v)$ and $E^-(v)$ denote the sets of all edges having tail $v$ or head $v$, respectively. A graph $G$ is {\it $k$-flowable} if all the edges of $G$ can be oriented
and assigned nonzero numbers  with absolute value less than $k$ so that
for every vertex $v\in V(G)$, the sum of the values on all the edges in $E^+(v)$ equals that of the values of
all the edges in $E^-(v)$. If $G$ is  $k$-flowable  we also say that $G$ admits  a nowhere-zero $k$-flow.

Let $A$ be an abelian group with  identity 0, and
let $A^*=A-\{0\}$.
Given an orientation and a mapping  $f: E(G)\rightarrow A$, the {\em boundary} of $f$ is a
function $\partial f: V(G)\rightarrow A$ defined by, for each vertex $v\in V(G)$,
\[
\partial f(v)=\sum\limits_{e\in E^{+}(v)} f(e)- \sum\limits_{e\in
E^-(v)} f(e),
\]
where ``$\sum$" refers to the addition in $A$.

A mapping $b: V(G)\rightarrow
A$ is a {\em zero-sum function} if $\sum\limits_{v\in V(G)} b(v)=0$.
A graph $G$ is {\em $A$-connected} if for every zero-sum function $b: V(G)\rightarrow
A$, there exists an
orientation of $G$ and a mapping  $f: E(G)\rightarrow A^*$ such that
$\partial f(v)=b(v)$ for each $v\in V(G)$.

The concept of $k$-flowability was first introduced by Tutte~\cite{Tu54}, and this theory  provides an interesting way to investigate the coloring of planar graphs in the sense that Tutte~\cite{Tu54} proved a classical theorem: a planar graph is $k$-colorable if and only if it is $k$-flowable. Jaeger {\em et al.}~\cite{Jaeger92} successfully generalized nowhere-zero flow problems to group connectivity. The purpose of study in group connectivity is to characterize contractible configurations for integer flow problems. Let $Z_3$ be a cyclic group of order three.
Obviously, if $G$ is $Z_3$-connected, then $G$ is
$3$-flowable.

An integer-valued sequence $\pi=(d_1, \ldots, d_n)$ is {\em graphic} if
there is a simple graph $G$ with degree sequence $\pi$. We say $\pi$ has  a {\em realization} $G$, and we also say $G$ is a realization of $\pi$.
If an integer-valued sequence $\pi$ has a realization $G$ which is  $A$-connected, then we say that $G$ is a $A$-connected realization of $\pi$ for an abelian group $A$. In particular, if $A=Z_3$,  then $G$ is a $Z_3$-connected realization of $\pi$, and we also say that $\pi$ has a $Z_3$-connected realization $G$.
In this paper, we write every degree sequence $(d_1,  \ldots, d_n)$ is in nonincrease order. For simplicity,
we use exponents to denote degree multiplicities, for example, we
write $(6, 5, 4^4, 3)$ for $(6, 5, 4, 4, 4, 4, 3)$.

The problem of realizing degree sequences by
graphs  that have nowhere-zero flows or are  $A$-connected, where $A$ is an abelian group, has been studied.  Luo {\em et al.} \cite{Zhang04}  proved that every bipartite graphic sequence with least element at least 2 has a 4-flowable realization. As a corollary, they confirmed the simultaneous edge-coloring conjecture of Cameron \cite{Ca}.
 Fan {\em et al.} \cite{Lai08} proved that every degree sequence with least element at least 2 has a realization which contains a spanning eulerian subgraph; such graphs is  4-flowable.
 Let $A$ be an abelian group with $|A|=4$.  For a nonincreasing $n$-element graphic sequence $\pi$ with least element at least 2 and sum at least $3n-3$,  Luo {\em et al.} \cite{Yu12} proved that $\pi$ has a realization that is $A$-connected. Yin and Guo \cite{Yin13} determined the smallest degree sum that yields graphic sequences with a $Z_3$-connected realization. For  the literature for this topic, the readers can see a survey \cite{LLSZ10}. In particular, Luo {\em et al.}\cite{Zhang08} completely answered the question of Archdeacon~\cite{DA}: Characterize all graphic sequences $\pi$  realizable by a 3-flowable graph. The natural group connectivity version of Archdeacon's problem is as follows.
\begin{prob}
\label{pro1}
Characterize all graphic sequences $\pi$  realizable by a $Z_3$-connected graph.
\end{prob}

On this problem,
 Luo {\em et al.} {\rm \cite{Zhang08}} obtained the next two results.

\begin{thm}
\label{th12}
Every nonincreasing graphic sequence $(d_1, \ldots, d_n)$  with $d_1=n-1$ and $d_n\ge 3$
 has a $Z_3$-connected realization unless $n$ is even and the sequence is $ (k, 3^k)$ or $(k^2, 3^{k-1})$, where $k=n-1$.
\end{thm}

\begin{thm}
\label{th13}
Every nonincreasing graphic sequence $(d_1,  \ldots, d_n)$  with
 $d_n\ge 3$ and $d_{n-3}\ge 4$ has a $Z_3$-connected realization.
\end{thm}

Motivated by Problem~\ref{pro1} and  the results above,  we present the following two theorems in this paper. These results extend the results of \cite{Zhang08}
by extending the characterizations to a large set of sequences.

\begin{thm}
\label{th1} A nonincreasing graphic sequence $(d_1,  \ldots, d_n)$  with
 $d_1\ge n-3$ and $d_n\geq 3$
 has a $Z_3$-connected realization
unless the sequence is $(n-3, 3^{n-1})$ for any $n$ or is $(k, 3^k)$ or $(k^2, 3^{k-1})$, where $k=n-1$ and $n$ is even.
\end{thm}

\begin{thm}
\label{th2}
A nonincreasing graphic sequence $(d_1, \ldots, d_n)$  with
 $d_n\ge 3$ and $d_{n-5}\ge 4$ has a $Z_3$-connected realization unless
the sequence is $(5^2, 3^4)$ or $(5, 3^5)$.
\end{thm}

We end this section with some notation and terminology. A graph is trivial if $E(G)=\emptyset$ and nontrivial otherwise. A {\em $k$-vertex} denotes a vertex of
degree $k$. Let $P_n$ denote the path on $n$
vertices and we call $P_n$ a $n$-path.
An {\em $n$-cycle} is a cycle on $n$ vertices.
The {\em wheel $W_k$} is the graph obtained from a $k$-cycle by
adding a new vertex, the center of the wheel,  and joining it to every vertex of the $k$-cycle.
 A wheel $W_k$ is an {\em odd (even)} wheel if  $k$ is odd (even). For
simplicity, we say $W_1$ is a triangle. For a graph $G$
and $X\subseteq V(G)$,
denote by $G[X]$ the subgraph of $G$ induced by $X$. For two vertex-disjoint subsets $V_1, V_2$ of $V(G)$,
denote by $e(V_1, V_2)$ the number of edges with one endpoint in
$V_1$ and the other endpoint in $V_2$.

We organize this paper as follows. In Section 2, we state some results and establish some lemmas that will be used in the following proofs. We will deal with some special degree sequences, each of which has a $Z_3$-connected realization in Section 3. In Sections 4 and 5, we will give the proofs of Theorems~\ref{th1} and~\ref{th2}.

\section{Lemmas}

Let $\pi=(d_1,  \ldots, d_n)$ be a graphic sequence with $d_1\ge  \ldots \ge d_n$.
Throughout this paper, we use $\bar{\pi}$ to represent the sequence $(d_1-1, \ldots, d_{d_n}-1,
d_{d_n+1}, \ldots, d_{n-1})$, which is called the
\emph{residual sequence} obtained from $\pi$ by deleting
$d_n$. The following well-known result is due to Hakimi \cite{SH, SLH} and Kleitman and Wang \cite{DK}.

\begin{thm}
\label{th20} A graphic sequence has even sum. Furthermore, a sequence $\pi$ is graphic if and only if $\bar{\pi}$ is graphic.
\end{thm}

Some results in  \cite{Chen08, DeVos06, Fan08,
Jaeger92, Lai00} on group connectivity are
summarized as follows.

\begin{lemma}
\label{le21} Let $A$ be an abelian group with $|A|\ge 3$. The following
results are known:

(1) $K_1$ is $A$-connected;

(2) $K_n$ and $K_n^-$ are $A$-connected if $n\ge 5$;

(3) An $n$-cycle is $A$-connected if and only if $|A|\ge n+1$;

(4) $K_{m,n}$ is $A$-connected if $m\ge n\ge 4$; neither $K_{2,t}$
$(t\ge 2)$ nor $K_{3,s}$ $(s\ge 3)$ is $Z_3$-connected;

(5) Each even wheel is $Z_3$-connected and each odd wheel is not;

(6) Let $H\subseteq G$ and $H$ be $A$-connected. $G$ is
$A$-connected if and only if $G/H$ is $A$-connected;

(7) If $G$ is not $A$-connected, then any spanning subgraph of $G$
is not $A$-connected.

(8) Let $v$ be not a vertex of
$G$. If $G$ is $A$-connected and $e(v, G)\ge 2$, then $G\cup \{v\}$
is $A$-connected.
\end{lemma}

Let $G$ be a graph having an induced path with three vertices $v,u, w$ in order. Let $G_{[uv, uw]}$ be the graph by deleting $uv$ and $uw$ and adding a new edge $vw$.  The following lemma was first proved by Lai in~\cite{Lai00} and reformulated by Chen {\it et al.} in \cite{Chen08}.

\begin{lemma}
\label{le22} Let $G$ be a graph with $u\in V(G)$, $uv, uw\in E(G)$ and $d(u)\geq 4$, and let $A$ be an abelian group with $|A|\geq 3$. If
$G_{[uv,uw]}$ is $A$-connected, then so is $G$.
\end{lemma}

A graph $G$ is {\em triangularly connected} if for every edge $e, f\in E$ there
exists a sequence of cycles $C_1, C_2, \ldots, C_k$ such that $e\in E(C_1)$,
$f\in E(C_k)$, and $|E(C_i)|\le 3$ for $1\le i\le k$, and $|E(C_j)\cap E(C_{j+1})|\neq \emptyset$
for $1\le j\le k-1$.

\begin{lemma}
\label{triangle}{\rm (\cite{DeVos06})}  A triangularly connected graph
 $G$ is $Z_3$-connected if $G$ has minimum degree at least 4 or
has  a nontrivial $Z_3$-connected subgraph.
\end{lemma}

\medskip

An orientation $D$ of $G$ is  a {\it modular 3-orientation} if
$|E^+(v)|-|E^-(v)|\equiv 0$ (mod 3) for every vertex $v\in V(G)$.  Steinberg and Younger \cite{SY89} established the following relationship.

\begin{lemma}
\label{le24}
A graph
$G$ is 3-flowable if and only if $G$ admits a
modular 3-orientation.
\end{lemma}

Let  $v$ be a 3-vertex in  a graph $G$,  and let $N(v)=\{v_1, v_2, v_3\}$. Denote by $G_{(v,v_1)}$  the  graph obtained from $G$ by deleting vertex $v$  and adding a new edge $v_2v_3$.  The following lemma is due to Luo {\it et al.} \cite{Luo08}.

\begin{lemma}
\label{le23} Let $A$ be an abelian group with $|A|\geq 3$, and let $b: V(G)\mapsto A$
be a zero-sum function with $b(v)\neq 0$. If $G_{(v,v_1)}$ is $Z_3$-connected, then there exists an
orientation $D$ of $G$ and a nowhere-zero mapping $f': E(G)\mapsto A$ such that $\partial f'=b$
under the orientation $D$ of $G$.
\end{lemma}

For any odd integer $k$, Luo {\em et al.} \cite{Zhang08} proved that no realization of the graphic sequence $(k, 3^k)$ and $(k^2, 3^{k-1})$ is  3-flowable. This yields the following lemma.

\begin{lemma}
\label{le230}
 If $k$ is odd, then neither $(k, 3^k)$ nor $(k^2, 3^{k-1})$ has a $Z_3$-connected realization.
\end{lemma}

Next we provide  $Z_3$-connected realizations for some degree sequences.

\setlength{\unitlength}{0.8mm}
\begin{picture}(30,50)
\put(-10,5){\begin{picture}(30,30)
       \put(10,15){\line(1,0){30}}
       \put(20,25){\line(1,0){10}}
       \put(10,35){\line(1,0){30}}
       \put(10,15){\line(0,1){20}}
       \put(40,15){\line(0,1){20}}
       \put(10,15){\line(2,1){20}}
       \put(10,15){\line(1,1){10}}
       \put(20,25){\line(-1,1){10}}
       \put(30,25){\line(1,1){10}}
       \put(40,15){\line(-1,1){10}}
       \put(10,15){\circle*{1.5}}
       \put(7,15){\makebox(0,0){$v_1$}}
       \put(20,25){\circle*{1.5}}
       \put(20,28){\makebox(0,0){$v_5$}}
       \put(10,35){\circle*{1.5}}
       \put(7, 35){\makebox(0,0){$v_3$}}
       \put(40,35){\circle*{1.5}}
       \put(43,35){\makebox(0,0){$v_4$}}
       \put(40,15){\circle*{1.5}}
       \put(43, 15){\makebox(0,0){$v_6$}}
       \put(30,25){\circle*{1.5}}
       \put(30, 28){\makebox(0,0){$v_2$}}
       \put(25,5){\makebox(0,0){(a) $\pi=(4^2, 3^4)$}}
       \end{picture}}

       \put(-10,5){\begin{picture}(30,30)
       \put(60,15){\line(1,0){30}}
       \put(70,25){\line(2,1){20}}
       \put(60,35){\line(1,0){30}}
       \put(60,15){\line(0,1){20}}
       \put(90,15){\line(0,1){20}}
       \put(60,15){\line(3,1){15}}
       \put(60,15){\line(1,1){10}}
       \put(70,25){\line(-1,1){10}}
       \put(70,25){\line(2,-1){20}}
       \put(60,15){\circle*{1.5}}
       \put(57,15){\makebox(0,0){$v_2$}}
       \put(70,25){\circle*{1.5}}
       \put(70,25){\line(1,-1){10}}
       \put(70,28){\makebox(0,0){$v_1$}}
       \put(60,35){\circle*{1.5}}
       \put(57, 35){\makebox(0,0){$v_3$}}
       \put(90,35){\circle*{1.5}}
       \put(93,35){\makebox(0,0){$v_4$}}
       \put(90,15){\circle*{1.5}}
       \put(93, 15){\makebox(0,0){$v_5$}}
       \put(80,15){\circle*{1.5}}
       \put(75,20){\circle*{1.5}}
        \put(80, 12){\makebox(0,0){$v_6$}}
       \put(72, 20){\makebox(0,0){$v_7$}}
       \put(75,5){\makebox(0,0){(b) $\pi=(5, 4, 3^5)$}}
       \end{picture}}

\put(-10,5){\begin{picture}(30,30)
       \put(110,15){\line(1,0){30}}
       \put(120,25){\line(2,1){20}}
       \put(110,35){\line(1,0){30}}
       \put(110,15){\line(0,1){20}}
       \put(140,15){\line(0,1){20}}
       \put(110,15){\line(1,1){10}}
       \put(120,25){\line(-1,1){10}}
       \put(120,25){\line(2,-1){12}}
        \put(132,19){\circle*{1.5}}
       \put(110,15){\circle*{1.5}}
       \put(107,15){\makebox(0,0){$v_5$}}
       \put(120,25){\circle*{1.5}}
       \curve (120,25,130,25)
       \curve (126,18,132,19)
       \curve (130,25,132,19)
       \curve (126,18,140,15)
       \curve (130,25,140,15)
        \curve (120,25,126,18)
        \put(123,18){\makebox(0,0){$v_6$}}
       \put(135,19){\makebox(0,0){$v_7$}}
        \put(130,27){\makebox(0,0){$v_8$}}
       \put(120,28){\makebox(0,0){$v_1$}}
       \put(110,35){\circle*{1.5}}
       \put(107, 35){\makebox(0,0){$v_4$}}
       \put(140,35){\circle*{1.5}}
       \put(143,35){\makebox(0,0){$v_3$}}
       \put(140,15){\circle*{1.5}}
       \put(143, 15){\makebox(0,0){$v_2$}}
       \put(130,25){\circle*{1.5}}
       \put(126,18){\circle*{1.5}}
       \put(125,5){\makebox(0,0){(c) $\pi=(6, 4, 3^6)$}}
       \end{picture}}

       \put(40,5){\begin{picture}(30,30)
       \put(110,15){\line(1,0){30}}
       \put(120,25){\line(2,1){20}}
       \put(110,35){\line(1,0){30}}
       \put(110,15){\line(0,1){20}}
       \put(140,15){\line(0,1){20}}
       \put(110,15){\line(1,1){10}}
       \put(120,25){\line(-1,1){10}}
        \put(132,19){\circle*{1.5}}
       \put(110,15){\circle*{1.5}}
       \put(107,15){\makebox(0,0){$v_5$}}
       \put(120,25){\circle*{1.5}}
       \curve (120,25,130,25)
       \curve (126,18,132,19)
       \curve (130,25,132,19)
       \curve (126,18,140,15)
       \curve (130,25,140,15)
        \curve (120,25,126,18)
        \curve (132,19,140,15)

        \put(123,18){\makebox(0,0){$v_6$}}
       \put(135,19){\makebox(0,0){$v_7$}}
        \put(130,27){\makebox(0,0){$v_8$}}
       \put(120,28){\makebox(0,0){$v_1$}}
       \put(110,35){\circle*{1.5}}
       \put(107, 35){\makebox(0,0){$v_4$}}
       \put(140,35){\circle*{1.5}}
       \put(143,35){\makebox(0,0){$v_3$}}
       \put(140,15){\circle*{1.5}}
       \put(143, 15){\makebox(0,0){$v_2$}}
       \put(130,25){\circle*{1.5}}
       \put(126,18){\circle*{1.5}}
       \put(125,5){\makebox(0,0){(d) $\pi=(5^2, 3^6)$}}

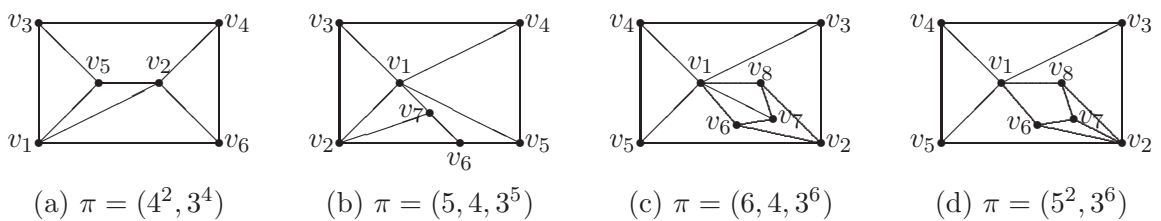
        \put(55,-5){\makebox(0,0){Fig. 1: Realizations of four degree sequences}}
       \end{picture}}
\end{picture}

\begin{lemma}
\label{fig2} Each of the graphs in Fig. 1 is $Z_3$-connected.
\end{lemma}
\begin{proof}
If $G$ is the graph (a) in Fig. 1, then $G$ is $Z_3$-connected by Lemma 2.2  of \cite{Luo08}. Thus, we may  assume that $G$ is one of the graphs (b), (c) and (d) shown in Fig 1.

We first prove here that the graph (b) in Fig. 1 is $Z_3$-connected.
Assume  that $G$ is the graph (b) shown in Fig. 1.
Let $b: V(G)\to Z_3$ be a zero-sum function.
If $b(v_3)\neq 0$, then $G_{(v_3,v_4)}$ contains a 2-cycle $(v_1,v_2)$.  Contracting this 2-cycle and repeatedly contracting all 2-cycles generated in the process, we
obtain $K_1$. By parts (1) and (6) of Lemma~\ref{le21}, $G_{(v_3,v_4)}$
is $Z_3$-connected. It follows by Lemma~\ref{le23} that there exists a nowhere-zero mapping $f: E(G)\to Z_3$ with
$\partial f=b$. Thus, we may assume that $b(v_3)=0$. Similarly, we may assume that $b(v_4)=b(v_5)=b(v_6)=b(v_7)=0$.
This means that for such  $b$, there are only three possibilities to be considered: $(b(v_1),b(v_2))\in \{(0,0), (1,2), (2, 1)\}$.

 If $(b(v_1),b(v_2))=(0,0)$, we show that $G$ is 3-flowable. Note that each vertex of $v_3, v_4, v_5, v_6$ and $v_7$ is of degree 3. The edges of $G$ are oriented as follows: $|E^+(v_3)|=3$, $|E^-(v_4)|=3$, $|E^+(v_5)|=3$,
$|E^-(v_6)|=3$, $|E^+(v_7)|=3$, and $v_2v_1$ is oriented from $v_2$ to $v_1$. It is easy to verify  that
$|E^+(v)|-|E^-(v)|=0$ (mod 3) for each vertex $v\in V(G)$. By Lemma~\ref{le24}, $G$ is 3-flowable. Thus, there is an $f: E(G)\to Z_3^*$ such that $\partial f(v)=b(v)$ for each $v\in V(G)$.

 If $(b(v_1), b(v_2)=(1, 2)$,  note that $b(v)=0$ for each $v\in V(G)-\{v_1, v_2\}$. The edges of $G$ are oriented as follows: $|E^-(v_3)|=3$,
 $|E^+(v_4)|=3$, $|E^-(v_5)|=3$,
$|E^+(v_6)|=3$, $|E^-(v_7)|=3$ and edge $v_2v_1$ is
 oriented from $v_2$ to $v_1$. If $(b(v_1), b(v_2)=(2, 1)$, then the edges of $G$ are oriented as follows: $|E^+(v_3)|=3$, $|E^-(v_4)|=3$, $|E^+(v_5)|=3$,
$|E^-(v_6)|=3$, $|E^+(v_7)|=3$ and edge $v_1v_2$ is
 oriented from $v_1$ to $v_2$.
 In each case,  for each $e\in E(G)$,
define $f(e)=1$.  It is easy to see that for $v\in \{v_3, v_4, v_5, v_6, v_7\}$, $\partial f(v)=0=b(v)$, $\partial f(v_1)=b(v_1)$ and $\partial f(v_2)=b(v_2)$.

Thus, for any zero-sum function $b$, there exists an orientation of $G$
and a nowhere-zero mapping $f: E(G)\to Z_3$ such that
$\partial f=b$. Therefore, $G$ is $Z_3$-connected.

Next, assume that
 $G$ is the graph (c) shown in Fig. 1.
Let $b: V(G)\to Z_3$ be a zero-sum function.
If $b(v_3)\neq 0$, then $G_{(v_3,v_4)}$ contains an even wheel $W_4$ induced by $\{v_1, v_2, v_6, v_7, v_8\}$ with the center $v_1$. Contracting this
$W_4$ and recursively contracting all 2-cycles generated in the process, we get
 $K_1$. By parts (1), (5) and (6) of Lemma~\ref{le21}, $G_{(v_3,v_4)}$
is $Z_3$-connected. By Lemma~\ref{le23}, there exists a nowhere-zero mapping $f: E(G)\to Z_3$ with
$\partial f=b$. Then we may assume $b(v_3)=0$. Similarly, $G_{(v_5,v_4)}$, $G_{(v_6,v_7)}$ and $G_{(v_8,v_7)}$
are all $Z_3$-connected. Thus,  we may assume $b(v_5)=b(v_6)=b(v_8)=0$.

If $b(v_4)\neq 0$, then $G_{(v_4,v_3)}$ contains a 2-cycle $(v_1,v_5)$. Contracting this
2-cycle, we get the graph containing an even wheel $W_4$ induced by $\{v_1, v_2, v_6, v_7, v_8\}$ with the center at $v_1$.
Contracting this wheel $W_4$ and contracting the 2-cycle generated in the process, we get  $K_1$.
By parts (1), (3), (5) and (6) of Lemma~\ref{le21}, $G_{(v_4,v_3)}$
is $Z_3$-connected. Lemma~\ref{le23} shows that there exists a nowhere-zero mapping $f: E(G)\to Z_3$ with
$\partial f=b$. Thus, we may assume $b(v_4)=0$. Similarly, $G_{(v_7,v_8)}$ is $Z_3$-connected. Therefore, we
may assume $b(v_7)=0$.

Thus, we are left to consider the  case that $b(v_i)=0$ for $3\leq i\leq 8$.
This implies that for such  function $b$, we need to consider three cases: $b(v_1)=b(v_2)=0$; $b(v_1)=1$ and $b(v_2)=2$; $b(v_1)=2$ and
$b(v_2)=1$.

In the case that $b(v_1)=b(v_2)=0$, we have $b(v)=0$ for each $v\in V(G)$. The edges of $G$ are oriented as follows: $|E^+(v_3)|=3$, $|E^-(v_4)|=3$, $|E^+(v_5)|=3$,
$|E^-(v_6)|=3$, $|E^+(v_7)|=3$ and $|E^-(v_8)|=3$. It is easy to verify that for each vertex $v\in V(G)$,
$|E^+(v)|-|E^-(v)|=0$ (mod 3). By Lemma~\ref{le24}, $G$ is 3-flowable. Thus, there is an $f: E(G)\to Z_3^*$ such that $\partial f(v)=b(v)$ for each $v\in V(G)$.

In the case that $b(v_1)=1$ and $b(v_2)=2$, we have $b(v)=0$ for each $v\in V(G)-\{v_1, v_2\}$. The edges of $G$ are oriented as follows:  $|E^+(v_3)|=3$,
 $|E^-(v_4)|=3$, $|E^+(v_5)|=3$,
$|E^+(v_6)|=3$, $|E^-(v_7)|=3$, $|E^+(v_8)|=3$. For each $e\in E(G)$,
define $f(e)=1$. It is easy to verify that
 $\partial f(v)=b(v)$ for each $v\in V(G)$.

In the case that $b(v_1)=2$ and $b(v_2)=1$, we have $b(v)=0$ for each $v\in V(G)-\{v_1, v_2\}$. The edges of $G$ are oriented as follows: $|E^-(v_3)|=3$, $|E^+(v_4)|=3$, $|E^-(v_5)|=3$,
$|E^-(v_6)|=3$, $|E^+(v_7)|=3$, $|E^-(v_8)|=3$. For each $e\in E(G)$,
define $f(e)=1$. It is easy to verify that
 $\partial f(v)=b(v)$ for each $v\in V(G)$.

Thus, for every zero-sum function $b$, there exists an orientation of $G$
and a mapping $f: E(G)\mapsto Z_3\setminus \{0\}$ such that
$\partial f=b$. Therefore, $G$ is $Z_3$-connected.

Finally, assume that $G$ is the graph (d) shown in Fig. 1.
Let $b: V(G)\to Z_3$ be a zero-sum function. As in the proof of the case when $G$ is the graph (c) in Fig. 1, we may assume that
that $b(v_i)=0$ for $3\leq i\leq 8$.
This implies that  for such  $b$, we need to consider only three cases: $b(v_1)=b(v_2)=0$; $b(v_1)=1$ and $b(v_2)=2$; $b(v_1)=2$ and $b(v_2)=1$.

In the case that $b(v_1)=b(v_2)=0$,  we have $b(v)=0$ for each $v\in V(G)$. Assume that the edges of $G$ are oriented as follows:  $|E^+(v_3)|=3$, $|E^-(v_4)|=3$, $|E^+(v_5)|=3$,
$|E^+(v_6)|=3$, $|E^-(v_7)|=3$, $|E^+(v_8)|=3$. It is easy to verify that for each vertex $v\in V(G)$,
$|E^+(v)|-|E^-(v)|=0$ (mod 3). By Lemma~\ref{le24}, $G$ is  3-flowable. Thus, there is an $f: E(G)\to Z_3\setminus \{0\}$ such that $\partial f(v)=b(v)$ for each $v\in V(G)$.

In the case that $b(v_1)=1$ and $b(v_2)=2$,  we have $b(v)=0$ for each $v\in V(G)-\{v_1, v_2\}$. Assume that the edges of $G$ are oriented as follows:  $|E^+(v_3)|=3$, $|E^-(v_4)|=3$, $|E^+(v_5)|=3$,
$|E^-(v_6)|=3$, $|E^+(v_7)|=3$, $|E^-(v_8)|=3$. For each $e\in E(G)$,
define $f(e)=1$. It is easy to verify that
 $\partial f(v)=b(v)$ for each $v\in V(G)$.

In the case that $b(v_1)=2$ and $b(v_2)=1$,  we have $b(v)=0$ for each $v\in V(G)-\{v_1, v_2\}$. Assume that the edges of $G$ are oriented as follows: $|E^-(v_3)|=3$, $|E^+(v_4)|=3$, $|E^-(v_5)|=3$,
$|E^+(v_6)|=3$, $|E^-(v_7)|=3$, $|E^+(v_8)|=3$. For each $e\in E(G)$,
define $f(e)=1$. It is easy to verify that
 $\partial f(v)=b(v)$ for each $v\in V(G)$.

Thus, for every zero-sum function $b$, there exists an orientation of $G$
and a mapping $f: E(G)\to Z_3\setminus \{0\}$ such that
$\partial f=b$. Therefore, $G$ is $Z_3$-connected.
\end{proof}

\begin{lemma}
\label{fig3} Each graph in Fig.2 is $Z_3$-connected.
\end{lemma}
\begin{proof}
  We first prove here that the graph (a) in Fig. 2 is $Z_3$-connected.  Denote by $G$ the graph (a)  in Fig. 2.
 We claim that $G$ is  3-flowable. Assume that the edges of the graph are oriented as follows:  $|E^+(v_4)|=3$, $|E^+(v_5)|=0$, $|E^+(v_6)|=3$, $|E^+(v_7)|=0$ and  $v_2v_3$ from $v_2$ to $v_3$. Define $f(e)=1$ for all $e\in E(G)$. It is easy to verify  that $\partial f(v)=0$ for each $v\in V(G)$. By Lemma~\ref{le24}, the graph (a) is  3-flowable.

\vskip 0.2cm
\vskip 0.2cm
\setlength{\unitlength}{0.8mm}
\begin{picture}(30,50)
\put(10,5){\begin{picture}(30,30)
       \put(10,15){\line(1,0){30}}
       \put(10,35){\line(1,0){30}}
       \put(10,15){\line(0,1){20}}
       \put(40,15){\line(0,1){20}}
       \put(10,15){\line(2,1){20}}
       \put(10,15){\line(1,1){10}}
       \put(40,15){\line(-1,1){10}}
       \put(25, 28){\circle*{1.8}}
       \curve (10,35,25,28)
        \curve (40,35,25,28)
        \curve (30,25,25,28)
        \curve (20,25,25,28)
        \curve (20,25,40,15)

       \put(10,15){\circle*{1.8}}
       \put(7,15){\makebox(0,0){$v_2$}}
      \put(20,25){\circle*{1.8}}
      \put(25,30){\makebox(0,0){$v_1$}}
       \put(10,35){\circle*{1.8}}
      \put(7, 35){\makebox(0,0){$v_4$}}
       \put(40,35){\circle*{1.8}}
      \put(43,35){\makebox(0,0){$v_5$}}
       \put(40,15){\circle*{1.8}}
      \put(43, 15){\makebox(0,0){$v_3$}}
       \put(30,25){\circle*{1.8}}
      \put(16, 25){\makebox(0,0){$v_6$}}
       \put(33, 25){\makebox(0,0){$v_7$}}
      \put(25,5){\makebox(0,0){(a) $\pi=(4^3, 3^4)$}}
       \end{picture}}

       \put(10,5){\begin{picture}(30,30)
       \put(60,15){\line(1,0){30}}
       \put(70,25){\line(2,1){20}}
       \put(60,35){\line(1,0){30}}
       \put(60,15){\line(0,1){20}}
       \put(90,15){\line(0,1){20}}
       \put(60,15){\line(3,1){15}}
       \put(60,15){\line(1,1){10}}
       \put(70,25){\line(-1,1){10}}
       \put(70,25){\line(2,-1){20}}
        \curve (75,20,70,15)
        \put(70,15){\circle*{1.8}}
       \put(60,15){\circle*{1.8}}
       \put(57,15){\makebox(0,0){$v_2$}}
       \put(70,25){\circle*{1.8}}
       \put(70,25){\line(1,-1){10}}
       \put(70,28){\makebox(0,0){$v_1$}}
       \put(60,35){\circle*{1.8}}
       \put(57, 35){\makebox(0,0){$v_7$}}
       \put(90,35){\circle*{1.8}}
       \put(93,35){\makebox(0,0){$v_8$}}
       \put(90,15){\circle*{1.8}}
       \put(93, 15){\makebox(0,0){$v_6$}}
       \put(80,15){\circle*{1.8}}
       \put(75,20){\circle*{1.8}}
        \put(80, 12){\makebox(0,0){$v_5$}}
       \put(71, 21){\makebox(0,0){$v_3$}}
        \put(70, 12){\makebox(0,0){$v_4$}}
       \put(75,5){\makebox(0,0){(b) $\pi=(5, 4^2, 3^5)$}}
       \end{picture}}

    \put(10,5){\begin{picture}(30,30)
       \put(110,15){\line(1,0){30}}
       \put(110,35){\line(1,0){30}}
       \put(110,15){\line(0,1){20}}
       \put(140,15){\line(0,1){20}}
       \put(110,15){\circle*{1.8}}
       \put(107,15){\makebox(0,0){$v_3$}}
       \put(125,30){\circle*{1.8}}
       \curve (125,30,110,15)
       \curve (125,30,140,15)

        \curve (125,30,110,35)
       \curve (125,30,140,35)

       \curve (125,20,110,15)
       \curve (125,20,140,15)

        \curve (125,20,110,28)
       \curve (125,20,140,28)
       \put(110,28){\circle*{1.8}}
       \put(140,28){\circle*{1.8}}
       \put(125,20){\circle*{1.8}}

       \put(125,16){\makebox(0,0){$v_2$}}
       \put(107,28){\makebox(0,0){$v_7$}}
       \put(147,28){\makebox(0,0){$v_8$}}
      \put(125,26){\makebox(0,0){$v_1$}}
      \put(110,35){\circle*{1.8}}
       \put(107, 35){\makebox(0,0){$v_5$}}
      \put(140,35){\circle*{1.8}}
      \put(143,35){\makebox(0,0){$v_6$}}
      \put(140,15){\circle*{1.8}}
      \put(143, 15){\makebox(0,0){$v_4$}}
      \put(125,5){\makebox(0,0){(c) $\pi=(4^4, 3^4)$}}
       \end{picture}}
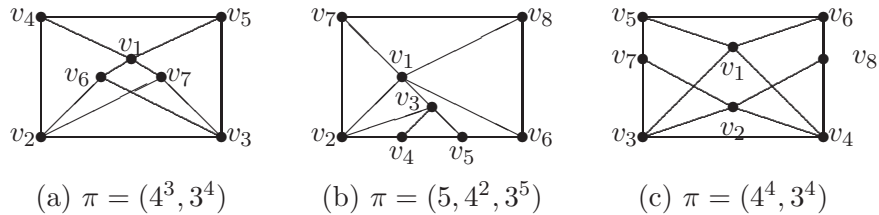
 \put(75,-5){\makebox(0,0){Fig. 2: Realizations of three degree sequences}}

\end{picture}

\vskip 0.8cm
Let $b: V(G)\to Z_3$ be a zero-sum function.
If $b(v_4)\neq 0$, then $G_{(v_4,v_2)}$ contains a 2-cycle $(v_1,v_5)$. Contracting the 2-cycle, we obtain an even wheel $W_4$ induced by $\{v_1, v_2, v_3, v_6, v_7\}$ with the center at $v_3$. By parts (3), (5) and (6) of Lemma~\ref{le21}, $G_{(v_4,v_2)}$
is $Z_3$-connected. By Lemma~\ref{le23}, there exists a nowhere-zero mapping $f: E(G)\to Z_3$ with
$\partial f=b$. Thus, we assume $b(v_4)=0$. By symmetry,  we may assume $b(v_5)=0$.
If $b(v_6)\neq 0$, then $G_{(v_6,v_3)}$ is a graph isomorphic to Fig 1 (a) which is $Z_3$-connected by Lemma~\ref{fig2}.
By Lemma~\ref{le23}, there exists a nowhere-zero mapping $f: E(G)\to Z_3$ with
$\partial f=b$. We thus assume $b(v_6)=0$. By symmetry,
we  assume $b(v_7)=0$.

So far, we may assume $b(v_4)=b(v_5)=b(v_6)=b(v_7)=0$.
 We claim that $b(v_2)\not=0$. If $b(v_2)=0$, then denote  by $G(v_2)$ the graph obtained from $G$ by deleting $v_2$ and adding edges $v_3v_7$ and $v_4v_6$. Contracting all 2-cycles, we finally get an even wheel $W_4$ with the center at $v_1$.   By Lemma~\ref{le21}, $G(v_2)$ is $Z_3$-connected. Thus, there exists a nowhere-zero mapping $f: E(G)\to Z_3$ with
$\partial f=b$. By symmetry,
we  assume that $b(v_3)\not=0$. Thus, we are left to discuss three cases $(b(v_1), b(v_2), b(v_3))\in \{ (1, 1, 1), (0, 1, 2), (2, 2, 2)\}$.

If $(b(v_1), b(v_2), b(v_3))=(1, 1, 1)$, then we orient the edges of $G$ as follows: $|E^+(v_4)|=3$, $|E^+(v_5)|=0$, $|E^+(v_6)|=3$, $|E^+(v_7)|=3$, and
 $v_2v_3$  from $v_2$ to $v_3$;
if $(b(v_1), b(v_2), b(v_3))=(0, 1, 2)$, then we orient edges of $G$ as follows: $|E^+(v_4)|=3$, $|E^+(v_5)|=0$, $|E^+(v_6)|=3$, $|E^+(v_7)|=0$, and
 $v_3v_2$ from $v_3$ to $v_2$;
if $(b(v_1), b(v_2), b(v_3))=(2, 2, 2)$, then we orient edges of $G$ as follows: $|E^+(v_4)|=3$, $|E^+(v_5)|=0$, $|E^+(v_6)|=0$, $|E^+(v_7)|=0$, and
 $v_2v_3$ from $v_2$ to $v_3$.
In each case, for each $e\in E(G)$ define $f(e)=1$.
It is easy to verify that $\partial f(v)=b(v)$ for each $v\in V(G)$.

In each case, there exists an orientation of $G$
and a nowhere-zero mapping $f: E(G)\to Z_3$ such that $\partial f=b$.
Therefore, $G$ is $Z_3$-connected.

Next we consider $G$ is the graph $(b)$ in Fig. 2. Let $b: V(G)\to Z_3$ be a zero-sum function.
If $b(v_4)\neq 0$, then $G_{(v_4,v_3)}$ is isomorphic the graph shown in Fig 1 (b).
By Lemma~\ref{fig2}, $G_{(v_4,v_3)}$
is $Z_3$-connected. By Lemma~\ref{le23}, there exists a nowhere-zero mapping $f: E(G)\to Z_3$ with
$\partial f=b$. Thus, we may assume $b(v_4)=0$. Similarly, if $b(v_5)\not=0$, we can prove that $G_{(v_5,v_3)}$
is $Z_3$-connected. Thus we also assume $b(v_5)=0$.

If $b(v_6)\neq 0$, then $G_{(v_6,v_5)}$ contains a 2-cycle $(v_1,v_8)$. Contracting
2-cycles obtained will result in a graph $K_1$.
By parts (1), (3) and (6) of Lemma~\ref{le21}, $G_{(v_6,v_5)}$
is $Z_3$-connected. By Lemma~\ref{le23}, there exists a nowhere-zero mapping $f: E(G)\to Z_3$ with
$\partial f=b$. Thus, we may assume $b(v_6)=0$. Similarly, if $b(v_7)\not=0$, we can prove that $G_{(v_7,v_8)}$, $G_{(v_8,v_6)}$ is $Z_3$-connected.
Thus, we may assume $b(v_7)=b(v_8)=0$.

From now on we may assume $b(v_4)=b(v_5)=b(v_6)=b(v_7)=b(v_8)=0$.
We claim that $b(v_2)\not=0$. If $b(v_2)=0$, then denote by $G(v_2)$ the graph obtained from $G$ by deleting $v_2$ and adding edges $v_1v_7$ and $v_3v_4$. By contracting all 2-cycles, finally we get $K_1$. By Lemma~\ref{le21}, $G(v_2)$ is $Z_3$-connected. Thus, for $b(v_2)=0$, there exists a nowhere-zero mapping $f: E(G)\to Z_3$ with
$\partial f=b$.
Similarly, we may assume that $b(v_3)\not=0$. Thus, it remains for us to discuss four cases $(b(v_1), b(v_2), b(v_3))\in \{ (1, 1, 1), (0, 1, 2), (2, 2, 2), (0, 2, 1)\}$.

If $(b(v_1), b(v_2), b(v_3))=(1, 1, 1)$, then we orient the edges of $G$  as follows: $|E^+(v_7)|=3$, $|E^+(v_8)|=0$, $|E^+(v_6)|=3$, $|E^+(v_5)|=0$, $|E^+(v_4)|=3$ and
$v_1v_2$ from $v_1$ to $v_2$,  $v_2v_3$  from $v_2$ to $v_3$, $v_1v_3$ from $v_1$ to $v_3$;
if $(b(v_1), b(v_2), b(v_3))=(0, 1, 2)$, then we orient the edges of $G$ as follows: $|E^+(v_7)|=3$, $|E^+(v_8)|=0$, $|E^+(v_6)|=3$, $|E^+(v_5)|=0$, $|E^+(v_4)|=3$ and
$v_2v_1$ from $v_2$ to $v_1$,  $v_3v_2$ from $v_3$ to $v_2$, $v_3v_1$ from $v_3$ to $v_1$;
if $(b(v_1), b(v_2), b(v_3))=(2, 2, 2)$, then we orient the edges of $G$ as follows: $|E^+(v_7)|=0$, $|E^+(v_8)|=3$, $|E^+(v_6)|=0$, $|E^+(v_5)|=3$, $|E^+(v_4)|=0$ and
 $v_2v_1$ from $v_2$ to $v_1$,  $v_3v_2$ from $v_3$ to $v_2$, $v_3v_1$ from $v_3$ to $v_1$;
if $(b(v_1), b(v_2), b(v_3))=(0, 2, 1)$, then we orient the edges of $G$ as follows: $|E^+(v_7)|=0$, $|E^+(v_8)|=3$, $|E^+(v_6)|=0$, $|E^+(v_5)|=3$, $|E^+(v_4)|=0$ and
$v_1v_2$ from $v_1$ to $v_2$,  $v_2v_3$ from $v_2$ to $v_3$, $v_1v_3$ from $v_1$ to $v_3$. In each case, for each $e\in E(G)$ define $f(e)=1$ in such orientation.
Clearly, we can verify that $\partial f(v)=b(v)$ for each $v\in V(G)$.
Thus, $G$ is $Z_3$-connected.

We are left to consider the case that $G$ is the graph $(c)$ in Fig. 2.  Let $b: V(G)\to Z_3$ be a zero-sum function.
If $b(v_5)\neq 0$, then $G_{(v_5,v_6)}$ contains an even wheel $W_4$  with the center at $v_3$.
By parts (5) and (8) of  Lemma~\ref{le21}, $G_{(v_5,v_6)}$
is $Z_3$-connected. By Lemma~\ref{le23}, there exists a nowhere-zero mapping $f: E(G)\to Z_3$ with
$\partial f=b$. Thus, we assume that $b(v_5)=0$. By symmetry, we may assume that $b(v_6)=0$.

If $b(v_7)\neq 0$, then $G_{(v_7,v_5)}$ contains a 2-cycle $(v_2,v_3)$. Contracting this
2-cycle and repeatedly contracting all 2-cycles generated in the process, we finally get $K_1$.
By parts (1) and (6) of Lemma~\ref{le21}, $G_{(v_7,v_5)}$
is $Z_3$-connected. By Lemma~\ref{le23}, there exists a nowhere-zero mapping $f: E(G)\to Z_3$ with
$\partial f=b$. Thus, we may assume $b(v_7)=0$. By symmetry, we may assume $b(v_8)=0$.

 If $b(v_1)=0$, then denote by $G(v_1)$ the graph obtained from $G$ by deleting $v_1$ and adding edges $v_5v_6$ and $v_3v_4$. Contracting all 2-cycles in the process, we finally get  $K_1$. By Lemma~\ref{le21}, $G(v_1)$ is $Z_3$-connected. Thus, for $b(v_1)=0$, there exists a nowhere-zero mapping $f: E(G)\to Z_3$ with
$\partial f=b$. Thus, we may assume that $b(v_1)\not=0$. Similarly, we also assume that  $b(v_2)\not=0$,  $b(v_3)\not=0$ and
$b(v_4)\not=0$. Thus, we only need to discuss four cases $(b(v_1), b(v_2), b(v_3), b(v_4))\in \{(2, 2, 1, 1), (2, 1, 2, 1), (1, 2, 1, 2), (1, 1, 2, 2)\}$ and $b(v_5)=b(v_6)=b(v_7)=b(v_8)=0$.

If $(b(v_1), b(v_2), b(v_3), b(v_4))=(2, 2, 1, 1)$, then we orient the edges of $G$ as follows: $|E^+(v_7)|=0$, $|E^+(v_5)|=3$, $|E^+(v_6)|=0$, $|E^+(v_8)|=3$, and
 $v_1v_3$ from $v_1$ to $v_3$, $v_1v_4$ from $v_1$ to $v_4$, $v_2v_3$ from $v_2$ to $v_3$, $v_2v_4$ from $v_2$ to $v_4$,  $v_4v_3$  from $v_4$ to $v_3$;
if $(b(v_1), b(v_2), b(v_3), b(v_4))=(2, 1, 2, 1)$, then we orient the edges of $G$ as follows: $|E^+(v_7)|=0$, $|E^+(v_5)|=3$, $|E^+(v_6)|=0$, $|E^+(v_8)|=3$, and
 $v_1v_3$ from $v_1$ to $v_3$, $v_1v_4$ from $v_1$ to $v_4$, $v_3v_2$ from $v_3$ to $v_2$, $v_4v_2$ from $v_4$ to $v_2$,  $v_3v_4$ from $v_3$ to $v_4$;
if $(b(v_1), b(v_2), b(v_3), b(v_4))=(1, 2, 1, 2)$, then we orient the edges of $G$ as follows: $|E^+(v_7)|=3$, $|E^+(v_5)|=0$, $|E^+(v_6)|=3$, $|E^+(v_8)|=0$, and
 $v_3v_1$ from $v_3$ to $v_1$, $v_4v_1$ from $v_4$ to $v_1$, $v_2v_3$ from $v_2$ to $v_3$, $v_2v_4$ from $v_2$ to $v_4$,  $v_4v_3$ from $v_4$ to $v_3$;
if $(b(v_1), b(v_2), b(v_3), b(v_4))=(1, 1, 2, 2)$, then we orient the edges of $G$ as follows: $|E^+(v_7)|=0$, $|E^+(v_5)|=3$, $|E^+(v_6)|=0$, $|E^+(v_8)|=3$, and
 $v_3v_1$ from $v_3$ to $v_1$, $v_4v_1$ from $v_4$ to $v_1$, $v_3v_2$ from $v_3$ to $v_2$, $v_4v_2$ from $v_4$ to $v_2$,  $v_4v_3$ from $v_4$ to $v_3$. In each case, for each $e\in E(G)$ define $f(e)=1$ in such orientation.
Clearly, we can verify that $\partial f(v)=b(v)$ for each $v\in V(G)$.
Thus, $G$ is $Z_3$-connected.
\end{proof}

\section{Some special cases}

Throughout this section, all sequences are graphic sequences. We provide $Z_3$-connected realizations for
 some graphic sequences.

\begin{lemma}
\label{31} Suppose that one of the following holds,

(i) $n\geq 6$ and $\pi=(n-2, 4, 3^{n-2})$;

(ii) $n\geq 5$ and  $\pi=(4^{n-4}, 3^4)$;

(iii) $n\geq 7$ and  $\pi=(5, 4^{n-6}, 3^5)$.

 Then $\pi$ has a $Z_3$-connected realization.
\end{lemma}
\begin{proof}
(i)  If $n=6$, then by Lemma~\ref{fig2}, $\pi$ has a
$Z_3$-connected realization $G$  in Fig. 1 (a). Thus, we assume that $n\ge 7$.

If $n=7,8$, then by Lemma~\ref{fig2}, $\pi$ has a $Z_3$-connected realization
$G$  in Fig. 1 (b) (c). Thus, we assume that $n\geq 9$.

Assume that $n$ is odd. Let $W_{n-5}$ be an even  wheel with the center at $v_1$
and $K_4^-$ on vertex set $\{u_1, u_2, u_3, u_4\}$ with $d_{K_4^-}(u_1)=d_{K_4^-}(u_3)=2$. Denote by  $G$  the graph obtained from $W_{n-5}$ and $K_4^-$ by
adding edges $u_iv_1$ for each $i\in \{1, 2, 3\}$.  Obviously, the graph $G$ has a degree sequence
$(n-2, 4, 3^{n-2})$. By part (5) of Lemma~\ref{le21}, $W_{n-5}$ is $Z_3$-connected.
The graph $G/W_{n-5}$ is  an even wheel $W_4$. By part (5) and (6) of Lemma~\ref{le21}, $G$ is $Z_3$-connected.
This means that $\pi$ has a $Z_3$-connected realization.

Assume that $n$ is even. Let $G_0$ be the graph in Fig. 1 (a)  and
$W_{n-6}$ be an even wheel with the center at $u_1$. Denote by $G$ the graph obtained from $W_{n-6}$ and $G_0$ by identifying $u_1$ and $v_1$.
Clearly, $G$ has a degree sequence $(n-2, 4, 3^{n-2})$. Since $n\ge 10$ is even, $W_{n-6}$ is
$Z_3$-connected by (5) of Lemma~\ref{le21}. By Lemma~\ref{fig2}, $G_0$ is $Z_3$-connected. This shows that $G$ is
$Z_3$-connected.

(ii)  If $n=5$, then an even wheel $W_4$ is a $Z_3$-connected realization of $\pi$;
if $n=6$, then
by Lemma~\ref{fig2}, $\pi$ has a
$Z_3$-connected realization $G$ in Fig. 1 (a); if $n=7$, then by Lemma~\ref{fig3}, $\pi$ has a $Z_3$-connected realization shown in Fig. 2 (a);
if $n=8$, then by Lemma~\ref{fig3}, the graph (c) in Fig. 2 is $Z_3$-connected realization of $\pi$.
If $n=9$, then let $G_1$ be an even wheel $W_4$ induced by $\{u_0, u_1, u_2, u_3, u_4\}$ with the center at $u_0$
and $G_2$ be a $K_4^-$ induced by $\{v_1, v_2, v_3, v_4\}$ with $d_{G_2}(v_1)=d_{G_2}(v_3)=3$. We construct a graph $G$ from $W_4$ and $K_4^-$ by
adding three edges $u_1v_2$, $u_2v_4$ and $u_3v_1$. Then $G$ is a $Z_3$-connected realization of
$(4^5, 3^4)$. Thus, we assume that $n\geq 10$.

Assume that $n=2k$, where $k\ge 5$. By induction of hypothesis,  let $G_i$ be a $Z_3$-connected realization of the degree sequence $(4^{k-4}, 3^4)$ for $i\in \{1, 2\}$.
Assume that $n=2k+1$, where $k\ge 5$. By induction hypothesis, let $G_1$ be a $Z_3$-connected realization of the degree sequence $(4^{k-4}, 3^4)$ and $G_2$
be a $Z_3$-connected realization of the degree sequence $(4^{k-3}, 3^4)$. In each case, we construct a graph $G$ from $G_1$ and $G_2$ by connecting
a pair of  3-vertices of $G_1$ to a pair of 3-vertices of $G_2$ one by one.
It is easy to verify that $G$ is a $Z_3$-connected realization of the degree sequence
$(4^{n-4}, 3^4)$.

(iii) If $n=7$, then by Lemma~\ref{fig2},  the graph (b) in
Fig. 1 is a $Z_3$-connected realization of $\pi$; if $n=8$, then by Lemma~\ref{fig3}, the graph (b) in Fig. 2 is
a $Z_3$-connected realization of $\pi$. If $n=9$, then $\pi=(5, 4^3, 3^5)$. Let $G_1$ be an even wheel $W_4$ induced by $\{u_0, u_1, u_2, u_3, u_4\}$ with the center at $u_0$
and $G_2$ be a $K_4^-$ induced by $\{v_1, v_2, v_3, v_4\}$ with $d_{G_2}(v_1)=d_{G_2}(v_3)=3$. We construct a graph $G$ from $W_4$ and $K_4^-$ by adding
 three edges $u_0v_2$, $u_1v_1$, $u_2v_4$. We conclude that $G$ is a $Z_3$-connected realization of  degree sequence
$(5, 4^3, 3^5)$. Thus, $n\ge 10$.

Assume that $n=2k$, where $k\ge 5$. By (ii), let $G_1$ and $G_2$ be  $Z_3$-connected realizations of
 degree sequence $(4^{k-4}, 3^4)$. Assume that $n=2k+1$, where $k\ge 5$. By (ii), let $G_1$ be a $Z_3$-connected realization of degree sequence $(4^{k-4}, 3^4)$ and $G_2$
be a $Z_3$-connected realization of  degree sequence $(4^{k-3}, 3^4)$.
In each case, choose one 4-vertex $u_1$ and one 3-vertex $u_2$ of $G_1$; choose two 3-vertices $v_1, v_2$ of $G_2$.
We construct a graph $G$ from $G_1$ and $G_2$ by adding $u_1v_1$ and $u_2v_2$.
Thus, $G$ is a $Z_3$-connected realization of degree sequence
$(5, 4^{n-6}, 3^5)$.
\end{proof}

\begin{lemma}
\label{34}
If $\pi=(n-3, 3^{n-1})$, then $\pi$ has not a $Z_3$-connected realization.
\end{lemma}
\begin{proof}
Suppose otherwise that $G$ has a $Z_3$-connected realization of degree sequence $(n-3, 3^{n-1})$.
Let $V(G)=\{u, u_1,  \ldots, u_{n-3}, x_1, x_2\}$, $N_G(u)=\{u_1, \ldots, u_{n-3}\}$ ($N$ for short),
  and $X=\{x_1, x_2\}$.
We now consider the following two cases.

\medskip

\n{\bf Case 1.} $x_1x_2\in E(G)$.

\medskip

Since $G$ is $Z_3$-connected, $G$ is 3-flowable.
By Lemma~\ref{le24} and symmetry, we may assume that $|E^+(x_1)|=3$ and $|E^-(x_2)|=3$. Since $d(u_i)=3$ for $i\in \{1, \ldots, n-3\}$, by Lemma~\ref{le24}, either $|E^+(u_i)|=3$  or  $|E^-(u_i)|=3$.
This implies that there exists no vertex $u_i$ in $N$ such that $u_ix_1, u_ix_2\in E(G)$.
Thus, $G[N]$ is   the union of two paths $P_1$ and $P_2$. We  relabel the vertices of $N$ such that $P_1=u_1\ldots u_k$ and $P_2=u_{k+1}\ldots u_{n-3}$.

Suppose first that $x_1u_1, x_1u_k\in E(G)$ and $x_2u_{k+1}, x_2u_{n-3}\in E(G)$.
Since $G$ is  3-flowable, by Lemma~\ref{le24}, $P_i$ contains odd  number of vertices for each $i\in \{1, 2\}$.
Define $b(u)=b(x_1)=b(x_2)=1$ and
$b(u_i)=0$ for each $i\in \{1, \ldots, n-3\}$. It is easy to verify that there exists no $f:E(G)\to Z_3^*$
such that $\partial f(v)=b(v)$ for each $v\in V(G)$, contrary to that $G$ is $Z_3$-connected.

Next, suppose that $x_1u_1, x_1u_{k+1}\in E(G)$ and $x_2u_k, x_2u_{n-3}\in E(G)$.
Since $G$ is  3-flowable, by Lemma~\ref{le24}  $P_i$ contains even number of vertices for each $i\in \{1, 2\}$.
Define $b(u)=1$, $b(x_2)=2$ and
$b(u_i)=b(x_1)=0$ for each $i\in \{1, \ldots, n-3\}$. It is easy to verify that there exists no $f:E(G)\to Z_3^*$
such that $\partial f(v)=b(v)$ for each $v\in V(G)$, contrary to that $G$ is $Z_3$-connected.

\medskip

\n{\bf Case 2.} $x_1x_2\notin E(G)$.

\medskip

Since $d(x_i)=3$ for each $i=1, 2$, $0\le |N(x_1)\cap N(x_2)|\le 3$.
Assume first that $|N(x_1)\cap N(x_2)|=3$. We assume, without loss of generality, that
 $u_1, u_2, u_3\in N(x_1)\cap N(x_2)$.  The subgraph induced by $\{u, x_1, x_2,u_1, u_2, u_3\}$ is $K_{3, 3}$ which is not $Z_3$-connected by part (4) of Lemma~\ref{le21}.
By part (6) of Lemma~\ref{le21}, $G$ is not $Z_3$-connected, a contradiction.

Assume that $|N(x_1)\cap N(x_2)|=2$. We assume, without loss of generality, that
 $u_1, u_2\in N(x_1)\cap N(x_2)$.  Since $G$ is 3-flowable, the graph $H$ induced by  $N\setminus \{u_1, u_2\}$ consists of even cycles and a path of length even.  This means that $n$ is even. If $n=6$, then this case cannot occur. Thus $n\ge 8$.
Define $b(x_1)=1$, $b(x_2)=2$ and
$b(u_i)=b(u)=0$ for each $i\in \{1,  \ldots, n-3\}$. It is easy to verify that there exists no $f:E(G)\to Z_3^*$
such that $\partial f(v)=b(v)$ for each $v\in V(G)$, contrary to that $G$ is $Z_3$-connected.

Next, assume that $|N(x_1)\cap N(x_2)|=1$. We assume, without loss of generality, that
$u_1\in N(x_1)\cap N(x_2)$. The graph induced by  $N\setminus \{u_1\}$ consists of even cycles and two paths $P_1$ and $P_2$.
Since $G$ is 3-flowable, $P_i$ contains odd vertices for each $i\in \{1, 2\}$.
Then $n$ is even. If $n=6, 8$, then this case cannot occur.
Thus, we assume that  $n\ge 10$.
Define $b(x_1)=1$, $b(x_2)=2$ and
$b(u_i)=b(u)=0$ for each $i\in \{1, \ldots, n-3\}$. In this case, there exists no $f:E(G)\to Z_3^*$
such that $\partial f(v)=b(v)$ for each $v\in V(G)$, contrary to that $G$ is $Z_3$-connected.

Finally, assume that $|N(x_1)\cap N(x_2)|=0$.
Then the graph induced by the vertices of $N$ consists of  three paths $P_1$, $P_2$ and $P_3$, together with even cycles.
We relabel the vertices of $N$ such that  $P_1=u_1\ldots u_s$, $P_2=u_{s+1}\ldots u_t$ and $P_3=u_{t+1}\ldots u_{n-3}$. By symmetry, we consider two cases: $x_1$ is adjacent to both the end vertices of some $P_i$; $x_1$ is adjacent to one of each $P_j$ for $j\in \{1, 2, 3\}$.

In the former case, we may assume that  $u_1x_1, u_sx_1\in E(G)$ and $x_2u_{s+1}, x_2u_t$. Since $G$ is 3-flowable, by Lemma~\ref{le24},  both $|V(P_1)|$ and $|V(P_2)|$ are
odd. If $|V(P_3)|$ is odd,
then define $b(x_1)=1$, $b(x_2)=2$ and
$b(u_i)=b(u)=0$ for each $i\in \{1, \ldots, n-3\}$.
If $|V(P_3)|$ is even,
then define $b(x_1)=1$, $b(x_2)=1$, $b(u)=1$ and
$b(u_i)=0$ for each $i\in \{1, \ldots, n-3\}$.
In either case,  there exists no $f:E(G)\to Z_3^*$
such that $\partial f(v)=b(v)$ for each $v\in V(G)$, contrary to that $G$ is $Z_3$-connected.

In the latter case,  $x_1u_1, x_1u_{s+1},x_1u_{t+1}, x_2u_s, x_2u_t, x_2u_{n-3}\in E(G)$.
It follows that $|V(P_1)|, |V(P_2)|$ and $|V(P_3)|$ have the same parity. If each of $|V(P_i)|$ for $i\in\{1, 2, 3\}$ is even, then
define $b(x_1)=1$, $b(x_2)=1$, $b(u)=1$ and
$b(u_i)=0$ for each $i\in \{1,  \ldots, n-3\}$. If each of $|V(P_i)|$ for $i\in\{1, 2, 3\}$ is odd, then define
define $b(x_1)=1$, $b(x_2)=2$, $b(u)=b(u_i)=0$
for each $i\in \{1,  \ldots, n-3\}$. In either case,  there exists no $f:E(G)\to Z_3^*$
such that $\partial f(v)=b(v)$ for each $v\in V(G)$, contrary to that $G$ is $Z_3$-connected.
\end{proof}

\vskip 0.8cm

\section{Proof of Theorem~\ref{th1}}

In order to prove Theorem~\ref{th1}, we establish the following lemma.

\begin{lemma}\label{41}
Suppose that $\pi=(d_1, \ldots, d_n)$ is a nonincreasing graphic sequence with $d_n\ge 3$.
If $d_1=n-2$, then $\pi$ has a $Z_3$-connected realization.
\end{lemma}
\begin{proof}
Suppose, to the contrary, that $\pi=(d_1, \ldots, d_n)$ has no $Z_3$-connected realization with
 $n$ minimized, where $d_1=n-2$.  By Theorem~\ref{th13}, we may assume
that $d_{n-3}\le 3$. In order to prove our lemma, we need the following claim.

\medskip

\n{\bf Claim 1.} Each of the following holds.

(i) $d_{n-3}=d_{n-2}=d_{n-1}=d_n=3$;

(ii) $n\ge 6$.

\n{\em Proof of Claim 1.} (i) follows since
$d_n\ge 3$.

(ii) Since $d_n=3$, $n\ge 4$. If $n=4$, then $d_1=3=n-1$,  contrary to that $d_1=n-2$. If $n=5$, then $\pi=(3^5)$  is not  graphic  by Theorem~\ref{th20}. This proves Claim 1.

\medskip

 If $n=6$, then $d_1=4$ and $d_3=d_4=d_5=d_6=3$. By Theorem~\ref{th20}, $d_2=4$. By Lemma~\ref{fig2}, $\pi=(4^2, 3^4)$ has a $Z_3$-realization, a contradiction. Thus, we may assume that $n\geq 7$.

\medskip

\n{\bf Claim 2.}  $d_3=3$.

\n{\em Proof of Claim 2.}
Suppose otherwise that $d_3\ge 4$ and $G$ is a counterexample with $|V(G)|=n$ minimized. Then $d_2\ge d_3\ge 4$. Hence,
$\bar{\pi}=(n-3, d_2-1, d_3-1, d_4, \ldots, d_{n-1})$=$(\bar{d_1},  \ldots, \bar{d}_{n-1})$
with $\bar{d_1}\ge  \ldots \ge \bar{d}_{n-1}$. This implies that $\bar{d}_{n-1}\ge 3$ and $\bar{d_1}=(n-1)-2$ or $\bar{d_1}=d_4$.
In the former case, since $\bar{d_1}=n-3=(n-1)-2$, by the minimality of $n$, $\bar{\pi}$ has a $Z_3$-connected realization $\bar{G}$.
In the latter case,  $\bar{d_1}\neq n-3$ and hence $\bar{d_1}=d_4>n-3$. Since $d_1=n-2\geq d_4$, $d_4=n-2$. This means that
$\bar{d_1}=d_4=(n-1)-1$. By Theorem~\ref{th12}, either $\bar{\pi}$ has a $Z_3$-connected realization $\bar{G}$ or $\bar{\pi}=(k, 3^k)$, $(k^2, 3^{k-1})$, where $k$ is odd.
If $\bar{\pi}=(k, 3^k)$, then $d_1=k+1=n-2$ and $d_2=d_3=4$. On the other hand, $n=k+1+1=k+2$. This contradiction proves that $\pi\not=(k, 3^k)$. Similarly,
 $\bar{\pi}\not=(k^2, 3^{k-1})$.
If $\bar{\pi}$ has a $Z_3$-connected realization $\bar{G}$, then
$\pi$ has a realization $G$ of $\pi$ from $\bar{G}$ by adding a new vertex $v$ and three edges joining $v$
to the corresponding vertices of $\bar{G}$. By part (8) of Lemma~\ref{le21}, $G$ is $Z_3$-connected, a contradiction.
Thus $d_3\le 3$. Clearly, $d_3\ge 3$. Then $d_3=3$. This proves Claim 2.

\medskip

By  Claims 1 and 2,   $\pi=(n-2, d_2, 3^{n-2})$. Since $\pi$ is graphic, $d_2$ is even whenever $n$ is even or odd. Moreover, $d_2\ge 4$.
Recall that $n\ge 7$. In this case, $\pi=(n-2, 4, 3^{n-2})$. By (i) of Lemma~\ref{31}, $\pi$ has a
$Z_3$-connected realization $G$, a contradiction. Thus, we may assume that
 $d_2\ge 6$. Since $n-2=d_1\geq d_2\geq 6$, $n\ge d_2+2\ge 8$.

Consider the case that $n$ is even. Denote by $W_{n-d_2+2}$ an even wheel with the center at $v_1$  and by $S$ a vertex set such that $|S|=d_2-4$ and $V(W_{n-d_2+2})\cap S=\emptyset$. Note that $|S|$ is even. We construct a graph $G$ from $W_{n-d_2+2}$ and $S$ as follows: First, pick two vertices $s_1, s_2$ of $S$ and
   add $(d_2-6)/2$ edges such that the subgraph induced by $S\setminus\{s_1, s_2\}$ is a perfect matching. Second, let $v_1$ connect to each vertex of $S$. Third, pick a vertex $v_2$ in $W_{n-d_2+2}$ and let $v_2$ join to each vertex of $S$.
Finally, add one new
vertex $x$  adjacent to $v_2$, $s_1$ and $s_2$.

We claim that $G$ has a degree sequence $(n-2, d_2, 3^{n-2})$. Since  $d_{W_{n-d_2+2}}(v_1)=n-d_2+2 \ge 4$,
 $d(v_1)=n-d_2+2+d_2-4=n-2$, $d(v_2)=3+d_2-4+1=d_2$, each vertex of $V(G)\setminus\{v_1, v_2\}$ is a 3-vertex.  Since $W_{n-d_2+2}$ is an even wheel,
by part (5)  of Lemma~\ref{le21}, this wheel is $Z_3$-connected. By part (8) of Lemma~\ref{le21}, $G$ is $Z_3$-connected,
a contradiction.

Consider the case that $n$ is odd.  Denote by $W_{n-d_2+1}$ an even wheel with the center at $v_1$  and by $S$ a vertex set with $|S|=d_2-3$ and $V(W_{n-d_2+1})\cap S=\emptyset$. We construct a graph $G$ from $W_{n-d_2+1}$ and $S$  as follows: First, let $v_1$ connect to each vertex of $S$. Second, pick one vertex $v_2$ in $W_{n-d_2+1}$ and  let $v_2$ join to each vertex of $S$. Third, add one
vertex $x$  adjacent to three vertices of $S$. Finally, add $(d_2-6)/2$ edges in $S$ so that the subgraph induced by vertices of $S$, each of
which is not adjacent to $x$, is a perfect matching.
We claim that $G$ is a realization of degree sequence $(n-2, d_2, 3^{n-2})$. Since $d_{W_{n-d_2+1}}(v_1)=n-d_2+1 \ge 4$,
 $d(v_1)=n-d_2+1+d_2-3=n-2$.
Note that $d(v_2)=3+d_2-3=d_2$, each vertex of $V(G)\setminus\{v_1, v_2\}$ is a 3-vertex.
Similarly, it can be verified that $G$ is a $Z_3$-connected realization of $\pi$,
a contradiction.
\end{proof}

\medskip

\n \textbf{Proof of Theorem~\ref{th1}.}
Assume that $\pi=(d_1, \ldots, d_n)$ is a nonincreasing graphic sequence with  $d_1\geq n-3$. If $\pi$ is one of $(n-3,3^{n-1})$, $(k, 3^k)$ and $(k^2, 3^{k-1})$, then by Lemmas~\ref{le230} and \ref{34}, $\pi$ has no $Z_3$-connected realization.

 Conversely, assume that $\pi\notin \{(n-3, 3^{n-1}), (k,3^k), (k^2, 3^{k-1})\}$. Since $d_1\geq n-3$ and $d_n\geq 3$, $n\geq 6$. In the case that $n=6$, by Theorem ~\ref{th12}, $d_1=3, 4$. If $d_1=3$, then $\pi=(3^6)$.  Since $n=6$, $(3^6)=(n-3, 3^{n-1})$, contrary to our assumption.  If $d_1=4$, then by (ii) of Lemma~\ref{31} $\pi$ has a $Z_3$-connected realization. In the case that $n=7$, by Theorems~\ref{th12} and \ref{th20}, $4\leq d_1\leq 5$. If $d_1=5$, then any realization of $\pi$ contains the graph (b) of fig. 1. By Lemma~\ref{fig2}, $\pi$ has a $Z_3$-connected realization. Assume that $d_1=4$. Since $n=7$, $(4,3^6)=(n-3, 3^{n-1})$. Thus, by our assumption,  $\pi\not=(4, 3^6)$. In this case,   any realization of $\pi$ contains the graph (a) in Fig. 2.  By Lemma~\ref{fig3}, $\pi$ has a $Z_3$-connected realization.  Thus, assume that $n\geq 8$.

 By Theorem~\ref{th12} and by Lemmas~\ref{34} and \ref{41}, we are left to prove
that if $d_{1}=n-3, d_n\ge 3$ and $d_2\neq 3$, then $\pi$ has a $Z_3$-connected realization.
Suppose otherwise that $\pi=(d_1,  \ldots, d_n)$ satisfying
\begin{equation}\label{eq1}
\mbox{$d_1=n-3$, $d_2\neq 3$, $d_n\ge 3$.}
\end{equation}
Subject to (\ref{eq1}),
\begin{equation} \label{eq2}
\mbox{ $\pi$ has no $Z_3$-realization with $n$ minimized.}
\end{equation}

  We establish the following claim first.

\n \textbf{Claim 1.} (i) $d_{n-3}=d_{n-2}=d_{n-1}=d_n=3$.

(ii)  $3 \le d_3\le 4$.

\n \emph{Proof of Claim 1.} By Theorem~\ref{th13},
 $d_{n-3}\le 3$. (i) follows since $d_n\ge 3$.

(ii)
Suppose otherwise that subject to (\ref{eq1}) and (\ref{eq2}), $\pi$ satisfies $d_3\ge 5$. Since $d_2\ge d_3$, $d_2\ge 5$. Define
$\bar{\pi}=(n-4, d_2-1, d_3-1, d_4, \ldots, d_{n-1})$=$(\bar{d_1},  \ldots, \bar{d}_{n-1})$
with $\bar{d_1}\ge \ldots \ge \bar{d}_{n-1}$.  Since $d_3\geq 5$, $d_2-1\geq d_3-1\geq 4$. This means that $\bar{d_1}\geq \bar{d_2}\geq 4$, and $\bar{d}_{n-1}\ge 3$. If $d_1>d_4$, then
$\bar{d_1}=(n-1)-3$. In this case,
 by the minimality of $n$, $\bar{\pi}$ has a $Z_3$-connected realization $\bar{G}$.
If $d_1=d_4$, then $\bar{d_1}=d_4$. It follows that $d_4>n-4$. This implies that $d_1=d_2=d_3=d_4=n-3$. Thus, $\bar{d_1}=d_4=n-3=(n-1)-2$.
By Lemma~\ref{41}, $\bar{\pi}$ has a $Z_3$-connected realization $\bar{G}$. In either case,
$\pi$ has a realization $G$ obtained from $\bar{G}$ by adding a new vertex $v$ and three edges joining $v$
to the corresponding vertices of $\bar{G}$.
By part (8) of Lemma~\ref{le21}, $G$ is $Z_3$-connected, a contradiction.
Thus $d_3\le 4$. Since $d_3\ge 3$, $3\le d_3\le 4$. This proves Claim 1.

\medskip

By Claim 1, we may assume that $\pi=(n-3, d_2, d_3, \ldots, d_{n-4}, 3^4)$ with $d_3\in \{3, 4\}$. We consider the following two cases.

\medskip

\n \textbf{Case 1.} $d_3=3$.

\medskip

In this case,  $\pi=(n-3, d_2, 3^{n-2})$. Since $\pi$ is graphic, $d_2$ is odd. Since $d_2\not=3$, $d_2\ge 5$.
We first assume that $d_2=5$. In this case,  $\pi=(n-3, 5, 3^{n-2})$.
If $n=8$, by Lemma~\ref{fig2}, the graph (d) in Fig. 1 is
a $Z_3$-connected realization of $\pi=(5^2, 3^6)$.
Thus,  assume that $n\ge 9$.

 Assume that $n$ is odd. Denote by  $W_{n-5}$ an even wheel with the center at $v_1$ and by $S$ a vertex set with $|S|=2$. We construct graph $G$ from $W_{n-5}$ and $S$ as follows:
First, connect $v_1$ to each vertex of $S$. Second, choose one vertex $v_2$ in $W_{n-5}$ and add two
vertices $x_1, x_2$ such that $x_i$ is adjacent to $v_2$ and each vertex of $S$
for each $i\in \{1, 2\}$.

Since $d(v_1)=n-5+2=n-3$, $d(v_2)=3+2=5$ and each vertex of $V(G)\setminus\{v_1, v_2\}$ is a 3-vertex, this means that
$G$ is a realization of degree sequence $(n-3, 5, 3^{n-2})$.
By part (5) of Lemma~\ref{le21}, $W_{n-5}$ is $Z_3$-connected. Note that
$G/W_{n-5}$ is  an even wheel $W_4$ which is also $Z_3$-connected by Lemma~\ref{le21}. It follows by part (6) of  Lemma~\ref{le21} that  $G$ is $Z_3$-connected,
a contradiction.

 Thus we may assume that $n$ is even. Denote by  $W_{n-6}$ an even wheel with the center at $v_1$ and let $S=\{s_1, s_2, s_3\}$ be a vertex set.
We construct graph $G$ from $W_{n-6}$ and $S$ as follows:  First, connect $v_1$ to each vertex of $S$. Second, choose one vertex $v_2$ in $W_{n-6}$ and  and let $v_2$ be adjacent to $s_1$. Finally, add two
vertices $x_1, x_2$ such that $x_1$ is adjacent to $v_2$ and  $s_2$, $s_3$;
$x_2$ is adjacent to each vertex of $S$.

It is easy to verify that  $d(v_1)=n-6+3=n-3$,
 $d(v_2)=3+2=5$ and each vertex of $V(G)\setminus\{v_1, v_2\}$ is a 3-vertex. This means that $G$ is a realization of degree sequence $(n-3, 5, 3^{n-2})$.
By (5) of Lemma~\ref{le21}, $W_{n-6}$ is $Z_3$-connected. By part (8) of Lemma~\ref{le21}, $W_{n-6}\cup  \{s_1\}$ is $Z_3$-connected.
Note that $G/\{W_{n-6}\cup  \{s_1\}\}$ is an even wheel $W_4$ which is $Z_3$-connected. It follows by (6) of Lemma~\ref{le21} that $G$ is $Z_3$-connected,
a contradiction.

 From now on, we assume that $d_2\ge 7$.  In this case, $n\ge d_2+3\ge 10$. Consider the case that $n$ is even. Denote by  $W_{n-d_2+1}$ an even wheel with the center at $v_1$ and by $S$ a vertex set with $|S|=d_2-4$.
We construct graph $G$ from $W_{n-d_2+1}$ and $S$  as follows: First, connect $v_1$ to each vertex of $S$. Second, pick one vertex $s$ of $S$ and let $S_1=S\setminus\{s\}$,
pick one vertex  $v_2$ in $W_{n-d_2+1}$ and connect $v_2$ to each vertex of $S_1$.
Third, pick two vertices $s_1, s_2$ of $S_1$, and add $(d_2-7)/2$ edges such that the induced subgraph by $S_1\setminus\{s_1, s_2\}$ is a perfect matching.
Finally, we add two
vertices $x_1$ and $x_2$ such that $x_i$ is adjacent to $v_2$, $x_i$ is adjacent to $s_i$ for $i=1, 2$ and $s$ is adjacent to each of $x_1$  and $x_2$.

 Since $d(v_1)=n-d_2+1+d_2-4=n-3$,  $d(v_2)=3+d_2-5+2=d_2$ and each vertex of $V(G)-\{v_1, v_2\}$ is a 3-vertex, this implies s that $G$ is a realization of degree sequence  $(n-3, d_2, 3^{n-2})$.
By (5) of Lemma~\ref{le21}, $W_{n-d_2+1}$ is $Z_3$-connected. Contracting this even wheel $W_{n-d_2+1}$ and contracting all 2-cycles generated in the process, we get  $K_1$. By (8) of Lemma~\ref{le21}, $G$ is $Z_3$-connected,
a contradiction.

Consider the case that $n$ is odd. Denote by  $W_{n-d_2}$ an even wheel with the center at $v_1$ and by $S$ a vertex set with $|S|=d_2-3$. We construct a graph from $W_{n-d_2}$ and $S$ as follows.
First, let $v_1$ be adjacent to each vertex of $S$. Second, pick two vertices $s_3$ and $s_4$ of $S$, define $S_1=S\setminus\{s_3, s_4\}$ and
pick one vertex $v_2$ in $W$ so that $v_2$ is adjacent to each vertex of $S_1$. Third, add two new
vertex $x_1, x_2$ such that $x_i$ is adjacent to each of  $v_2$, $s_3$ and $s_{4}$.  Finally, add $(d_2-5)/2$ edges in $S_1$ so that the subgraph induced by vertices of $S_1\setminus\{s_3, s_4\}$
is a perfect matching.

 Since $d(v_1)=n-d_2+d_2-3=n-3$, $d(v_2)=3+d_2-5+2=d_2$, and each vertex of $V(G)\setminus\{v_1, v_2\}$ is a 3-vertex,  $G$ is a realization of degree sequence  $(n-3, d_2, 3^{n-2})$. Similarly,
by parts (5) and (8) of Lemma~\ref{le21}, $G$ is $Z_3$-connected,
a contradiction.

\medskip

\n \textbf{Case 2.} $d_3=4$.

\medskip

In this case, $d_2\ge 4$ and $\pi =(n-3, d_2, 4, d_4, \ldots, d_{n-4}, 3^4)$.
Define $\bar{\pi}=(n-4, d_2-1, 3, d_4, \ldots, d_{n-4}, \\3^3)=(\bar{d}_1, \ldots, \bar{d}_{n-1})$ with $\bar{d}_1\ge \ldots \ge \bar{d}_{n-1}$. If $d_1=d_4$, then $n-3=4$ and hence $n=7$, contrary to assumption that $n\geq 8$. Thus, $d_1>d_4$. In this case, $\bar{d}_1=n-4$.

\medskip

 \n{\bf Claim 2.} $d_2=4$.

 \medskip

\n{\em Proof of Claim 2.}  Suppose otherwise that $d_2\ge 5$. Then $\bar{d_2}\ge 4$ and $\bar{\pi}$ satisfies (\ref{eq1}). By the minimality of $n$, $\bar{\pi}$ has a $Z_3$-connected realization $\bar{G}$.
Thus,  we conclude that $G$ is a $Z_3$-connected realization of $\pi$ obtained from $\bar{G}$ by adding a new vertex $v$ and three edges joining $v$
to the corresponding vertices of $\bar{G}$.
This contradiction proves Claim 2.

\medskip

By Claim 2, $d_2=4$. Assume that $i\in \{3, \ldots, n-4\}$ such that $d_i=4$ and $d_{i+1}=3$.
Thus $\pi=(n-3, 4^{i-1}, 3^{n-i})$.

\medskip

\n{\bf Claim 3.} $i=3$.

\medskip

\n{\em Proof of Claim 3.} If $i$ is even, then $n-i$ is odd (even) when $n$ is odd (even).
No matter whether  $n$ is odd or even,  there are odd vertices of odd degree, a contradiction. Thus, $i$ is odd.
If $i\ge 5$, then $\bar{\pi}=(n-4, 4^{i-3}, 3^{n-i+1})$ satisfies (\ref{eq1}). Recall that $n\ge 8$,
 by the minimality of $n$, $\bar{\pi}$ has a
$Z_3$-connected realization $\bar{G}$. In this case, we can obtain a realization $G$ of $\pi$ from $\bar{G}$
by adding a new vertex $v$ and three edges joining $v$
to the corresponding vertices of $\bar{G}$.
By (8) of Lemma~\ref{le21}, $G$ is $Z_3$-connected, a contradiction.
 This proves Claim 3.

\medskip

By Claim 3, $i=3$. This leads to that $\pi=(n-3, 4^2, 3^{n-3})$. Recall that $n\geq 8$.
If $n=8$, then by Lemma~\ref{fig3}, $\pi$ has a $Z_3$-connected realization. Thus, we may assume that $n\geq 9$.

In the case that $n$ is odd,  denote by $W_{n-5}$  the even wheel with the center at $v_1$ and by $S$ a vertex set with $|S|=2$. We construct a graph $G$ from $W_{n-5}$ and $S$ as follows. First, let $v_1$ be adjacent to each vertex of $S$. Second, pick two vertices $v_2, v_3$ in $W_{n-5}$ and add two
vertices $x_1, x_2$ such that $x_i$ is adjacent to $v_{i+1}$ and each vertex of $S$
for each $i\in \{1, 2\}$.

It is easy to verify that  $d(v_1)=n-5+2=n-3$, $d(v_i)=3+1=4$ for each $i\in \{2,3\}$, and
each vertex of $V(G)-\{v_1, v_2, v_3\}$ is a 3-vertex.
Obviously,
$G$ has a degree sequence $(n-3, 4^2, 3^{n-3})$.
By (5) of Lemma~\ref{le21}, $W_{n-5}$ is $Z_3$-connected.
$G/W_{n-5}$ is an even wheel $W_4$ which is $Z_3$-connected.  By (6) of Lemma~\ref{le21}, $G$ is $Z_3$-connected,
a contradiction.

In the case that $n$ is even,  denote by $W_{n-6}$  the even wheel with the center at $v_1$ and let $S=\{s_1, s_2, s_3\}$. We construct a graph $G$ from $W_{n-6}$ and $S$ as follows. First, let $v_1$ be adjacent to each vertex of $S$. Second, pick two vertices $v_2, v_3$ in $W_{n-6}$ so that $v_2$ is adjacent to $s_1$. Finally, we add two
vertices $x_1, x_2$ such that $x_1$ is adjacent to $v_3$ and  $s_2$, $s_3$ and such that
$x_2$ is adjacent to each vertex of $S$.

It is easy to verify that $d(v_1)=n-6+3=n-3$, $d(v_i)=3+1=4$ for each $i\in \{2, 3\}$, and each vertex of $V(G)-\{v_1, v_2, v_3\}$ is a 3-vertex. Obviously,
$G$ has a degree sequence $(n-3, 4^2, 3^{n-3})$.
By (5) Lemma~\ref{le21}, $W_{n-6}$ is $Z_3$-connected. By (8) of Lemma~\ref{le21}, $W_{n-6}\cup \{s_1\}$ is $Z_3$-connected.
$G/\{W_{n-6}\cup \{s_1\}\}$ is  an even wheel $W_4$ which is $Z_3$-connected. By (6) of Lemma~\ref{le21},  $G$ is $Z_3$-connected
a contradiction. $\blacksquare$

\section{Proof of Theorem~\ref{th2}}

We first establish the following lemma which is used in the proof of Theorem~\ref{th2}.
\begin{lemma}
\label{51}
Let $\pi=(d_1,  \ldots, d_n)$ be a nonincreasing graphic sequence.
If $d_n\ge 3$ and $d_{n-4}\ge 4$, then either $\pi$ has a $Z_3$-connected
realization or $\pi= (5^2, 3^4)$.
\end{lemma}
\begin{proof}
Since $d_{n-4}\geq 4$, $n\geq 5$. If $n=5$, then by Theorem~\ref{th13}, $\pi=(4, 3^4)$. In this case, an even wheel $W_4$ is a $Z_3$-connected realization of $\pi$. If $n=6$, then by Theorem~\ref{th13}, $\pi=(4^2, 3^4)$ or $(5^2, 3^4)$. If $\pi=(4^2, 3^4)$, then by Lemma~\ref{fig2}, the graph (a) shown in Fig.1 is a $Z_3$-connected realization of $\pi$. If $n=7$, then by Theorem~\ref{th13}, $\pi=(5^2, 4, 3^4)$. Let $G$ be the graph (b) shown in Fig. 1 which has degree sequence  $\pi=(5, 4, 3^5)$. Denote by $G'$ the graph obtained from $G$ by adding an edge joining a vertex of degree 3 to a vertex of degree 4. By Lemma~\ref{fig2}, $G$ is a $Z_3$-connected realization of $(5, 4, 3^4)$ and so $G'$ is a $Z_3$-connected realization of $(5^2, 4, 3^4)$. Thus, assume that $n\geq 8$.

By Theorems~\ref{th13} and ~\ref{th1}, it is sufficient to prove that
if $d_{n-3}=3$, $d_{n-4}\ge 4$ and $d_1\le n-4$, then
$\pi$ has a $Z_3$-connected realization. In this case, $\pi=(d_1,  \ldots, d_{n-4}, 3^4)$.
 Suppose, to the contrary, that $\pi=(d_1,  \ldots, d_n)$ satisfies
\begin{equation}\label{eq3}
\mbox {$d_{n-3}=3$, $d_{n-4}\ge 4$ and $d_1\le n-4$.}
\end{equation}
Subject to (\ref{eq3}),
\begin{equation}
\label{eq4}
\mbox{$\pi$ has no $Z_3$-connected realization  with  $n$ minimized.}
\end{equation}

Assume that $d_{n-4}\ge 5$. Define $\bar{\pi}=(d_1-1, d_2-1, d_3-1, d_4, \ldots, d_{n-4}, 3^3)
=(\bar{d_1},  \ldots, \bar{d}_{n-1})$. Since $d_1, d_2, d_3\ge 5$ and $n\geq 8$,
$\bar{d}_{n-5}\ge 4$. This implies that $\bar{\pi}$ satisfies (\ref{eq3}), by the minimality of $n$, $\bar{\pi}$ has a $Z_3$-connected realization $\bar{G}$.
Thus, we construct  a realization $G$ of $\pi$ from $\bar{G}$ by adding a new vertex $v$ and three edges joining $v$
to the corresponding vertices of $\bar{G}$. It follows by (8) of Lemma~\ref{le21} that  $G$ is $Z_3$-connected, a contradiction.
Thus, we may assume that $d_{n-4}=4$.

On the other hand, if $d_1=4$, then $\pi=(4^{n-4}, 3^4)$. By Lemma~\ref{31},
$\pi$ has a $Z_3$-connected realization,
a contradiction. Thus, assume $d_1\ge 5$. Since $d_{n-4}=4$ and $n\geq 8$, $d_2\geq 5$ or $d_2=4$.

In the former case,  $\bar{d}_{n-5}\ge 4$. In this case, $\bar{\pi}$ satisfies (\ref{eq3}),
by the minimality of $n$,
$\bar{\pi}$ has a $Z_3$-connected realization $\bar{G}$.
Thus, we can construct a realization $G$ of $\pi$ from $\bar{G}$ by adding a new vertex $v$ and three edges joining $v$
to the corresponding vertices of $\bar{G}$. By (8) of Lemma~\ref{le21}, $G$ is $Z_3$-connected, a contradiction.

In the latter case,  $\pi=(d_1, 4^{n-5}, 3^4)$.
Since $\pi$ is graphic, $d_1$ is even. Since $d_1\le n-4$, $n-d_1-1\ge 3$.
In the case that $n-d_1-1=3$, we have $d_1=n-4$ and $n\ge 10$ is even. Denote by  $W_{n-4}$ an even wheel with the center at $v_1$. We construct a graph $G$ from $W_{n-4}$ as follows. First,
choose five vertices $v_2, v_3, v_4, v_5, v_6$ of $W_{n-4}$. Second, add three vertices $x_1, x_2, x_3$ and
edges $x_1x_2$, $x_2x_3$. Finally, add edges $v_2x_1, v_3x_1, v_4x_2, v_5x_3, v_6x_3$. In this case, for each $i\in \{1, 2, 3\}$, $x_i$ is a 3-vertex, and  for each $j\in \{ 2,  \ldots, 6\}$ $v_j$   is a 4-vertex.

It is easy to see that $G$ is a $Z_3$-connected realization of degree sequence $(n-4, 4^5, 3^{n-6})$.
Define $S=V(W_{n-4})\setminus\{v_2, \ldots, v_6\}=\{v_1, v_7,  \ldots, v_{n-3}\}$. Then $|S|=n-4-5=n-9\ge 1$.
If $n=10$, then $G$ is a realization of $(6, 4^{5}, 3^4)$. If $n\ge 12$, then define $G'$ from $G$ by adding $v_jv_{n-j+4}$ for $7\leq j\leq \frac{n}{2}+1$, that is, adding $(n-10)/2$ edges
in $S$. Obviously,
$G'$ has a degree sequence $(n-4, 4^{n-5}, 3^4)$. We conclude that $G'$ is  a $Z_3$-connected
realization of $\pi$.

In the case that $n-d_1-1=4$, we have $d_1=n-5$ and $n\ge 11$ is odd. Denote by  $W_{n-5}$ an even wheel with the center at $v_1$.
We construct a graph $G$ from $W_{n-5}$ as follows. First, choose six vertices $v_2, v_3, v_4, v_5, v_6, v_7$ of $W_{n-5}$. Second,
add four vertices $x_1, x_2, x_3, x_4$ and
edges $x_1x_2$, $x_2x_3$, $x_3x_4$. Third, add edges $x_1v_2, x_1v_3, x_2v_4, x_3v_5, x_4v_6,x_4v_7$. In this case, for each $i\in \{1, 2, 3, 4\}$, $x_i$ is a 3-vertex, and  for each $j\in \{ 2, 3, \ldots, 7\}$
$v_j$   is a 4-vertex.

 It is easy to see that $G$ is a  $Z_3$-connected realization of
degree sequence $(n-5, 4^6, 3^{n-7})$.
Define $S=V(W_{n-5})\setminus\{v_1, v_2, \ldots, v_7\}=\{v_8, \ldots, v_{n-4}\}$. Then $|S|=n-5-6=n-11\ge 0$.
If $n=11$, then $G$ is a realization of $(6, 4^{6}, 3^4)$. If $n\ge 13$, then define $G'$ from $G$ by  adding edges $v_jv_{n-j+4}$ for $8\leq j\leq \frac{n+3}{2}$, that is, adding $(n-11)/2$ edges
in $S$. Obviously,
$G'$ is a $Z_3$-connected realization of  degree sequence $(n-5, 4^{n-5},3^4)$, a contradiction.

In the case that $n-d_1-1\ge 5$,
denote by $W_{d_1}$  an even wheel with the center at $v_0$. Let $V(W_{d_1})=\{v_0, v_1,  \ldots, v_{d_1}\}$.
  Let $C:u_1\ldots u_{n-d_1-1}u_1$ be a cycle of length $n-d_1-1$ and
and define a graph $H$ obtained from $C$ adding edges $u_iu_{i+2}$ for each $i\in \{1,  \ldots, n-d_1-1\}$, where the subscripts are taken modular $n-d_1$.
Clearly, $H$ is a 4-regular and is triangularly connected.
Define $H'=H-\{u_{2}u_{n-d_1-1}\}$. Now we prove $H'$ is $Z_3$-connected. Clearly, $H'_{[u_1u_2,u_1u_3]}$
is triangularly connected and contains a 2-circuit $u_2u_3u_2$. By Lemma~\ref{triangle} (a) and
Lemma~\ref{le21} (3), $H'_{[u_1u_2, u_1u_3]}$ is $Z_3$-connected, and hence $H'$ by Lemma~\ref{le22}. We construct  a graph $G$ from $W_{d_1}$ and $H'$ as follows.
If $d_1\ge 8$, then we add two edges $v_1u_{n-d_1-1}$, $v_2u_2$
and add edges $v_jv_{d_1-j+3}$ for $3\leq j\leq \frac{d_1}{2}-1$, that is, add $(d_1-6)/2$ edges between vertices $\{v_3,  \ldots, v_{d_1-4}\}\setminus \{v_{\frac{d_1}{2}}, v_{\frac{d_1}{2}+1}, v_{\frac{d_1}{2}+2}, v_{\frac{d_1}{2}+3}\}$ such that $d(v_i)=4$ for each vertex of $\{v_3,  \ldots, v_{d_1-4}\}\setminus \{v_{\frac{d_1}{2}}, v_{\frac{d_1}{2}+1}, v_{\frac{d_1}{2}+2}, v_{\frac{d_1}{2}+3}\}$ and the new graph is simple.
If $d_1=6$, then we add two edges $v_1u_{n-d_1-1}$, $v_2u_2$.
In either case, $G$ is a $Z_3$-connected realization of
has a degree sequence $(d_1, 4^{n-5}, 3^4)$, a contradiction
\end{proof}

\medskip

\n \textbf{Proof of Theorem~\ref{th2}.} If $\pi=(5^2, 3^4)$ or $(5, 3^5)$, then by Lemma~\ref{le230}, $\pi$ has no $Z_3$-connected realization. Thus, assume that $\pi\not=(5^2, 3^4), (5, 3^5)$.

Since $d_{n-5}\geq 4$, $n\geq 6$. If $n=6$, then by Theorem~\ref{th13}, $\pi=(d_1, d_2, 3^4)$. By Lemma~\ref{51}, $\pi=(d_1, 3^5)$. By our assumption that $d_{n-5}\geq 4$, $\pi=(5, 3^5)$, a contradiction. If $n=7$, then by Theorem~\ref{th13} and Lemma~\ref{51}, $\pi=(d_1, d_2, 3^5)$. By our assumption that $d_{n-5}\geq 4$, $\pi=(5,4, 3^5)$ or $(6, 5, 3^5)$. In the former, the graph (b) in Fig. 1 is a $Z_3$-connected realization of $\pi$. In the latter case, Theorem~\ref{th1} shows that $\pi$ has a $Z_3$-connected realization. Thus, we may  assume that $n\geq 8$.

By Theorems~\ref{th13} and \ref{th1}, and Lemma~\ref{51},
it is sufficient to prove that if $d_{n-4}=3$, $d_{n-5}\ge 4$ and $d_1\le n-4$, then
$\pi$ has a $Z_3$-connected realization. Then $\pi=(d_1,  \ldots, d_{n-5}, 3^5)$.
 Suppose to the contrary that $\pi$ satisfies
\begin{equation}\label{eq5}
\mbox{ $d_{n-4}=3$, $d_{n-5}\ge 4$ and $d_1\le n-4$.}
\end{equation}
Subject to (\ref{eq5}),
\begin{equation}\label{eq6}
\mbox{$\pi$ has no $Z_3$-connected realization  with  $n$ minimized.}
\end{equation}

We claim that $d_3=4$. Suppose otherwise that  $d_{3}\ge 5$. Define $\bar{\pi}=(d_1-1, d_2-1, d_3-1, d_4, \ldots, d_{n-5}, 3^4)
=(\bar{d}_1, \ldots, \bar{d}_{n-1})$. Since $d_1, d_2, d_3\ge 5$,
$\bar{d}_{n-6}\ge 4$. Thus $\bar{\pi}$ satisfies (\ref{eq5}). By the minimality of $n$, $\bar{\pi}$ has a $Z_3$-connected realization $\bar{G}$. Denote by $G$ the graph obtained from $\bar{G}$ by adding a new vertex $v$ and three edges joining $v$
to the corresponding vertices. It follows by (8) of Lemma~\ref{le21} that $G$ is a $Z_3$-connected realization of $\pi$, a contradiction.
Thus,  $d_3=4$ and $\pi=(d_1, d_2, 4^{n-7}, 3^5)$.

We claim that $d_2=4$. Suppose otherwise that  $d_2\ge 5$.  In this case, $\bar{\pi}=(d_1-1, d_2-1, 4^{n-8}, 3^6)
=(\bar{d}_1,  \ldots, \bar{d}_{n-1})$. Since $d_2\ge 5$ and $d_3=4$, $n\ge 3+5=8$. This implies that
$\bar{d}_{n-6}\ge 4$. Thus, $\bar{\pi}$ satisfies (\ref{eq5}). By the minimality of $n$,
$\bar{\pi}$ has a $Z_3$-connected realization $\bar{G}$. Denote by $G$ the graph obtained
 from $\bar{G}$ by adding a new vertex $v$ and three edges joining $v$
to the corresponding vertices of $\bar{G}$. It follows from (8) of Lemma~\ref{le21} that $G$ is a $Z_3$-connected realization of $\pi$, a contradiction.
Thus, $d_2=4$ and $\pi=(d_1, 4^{n-6}, 3^5)$.

Since $\pi$ is graphic, $d_1$ is odd and $d_1\ge 5$.
If $d_1=5$, then by (iii) of Lemma~\ref{31}, $\pi$ has a $Z_3$-connected realization.

We are left to the case that $d_1\ge 7$. Since $d_1\le n-4$, $n\ge d_1+4 \ge 11$. By (iii) of
Lemma~\ref{31}, Let $G'$ be a $Z_3$-connected realization of
degree sequence $(5, 4^{n-6}, 3^5)$. By the construction of $G'$ in (iii) of Lemma~\ref{31},
$G'$ has at least $|E(G')|-5-14=2n-21\ge 2(d_1+4)-21=2d_1-13\ge (d_1-5)/2$ edges
not incident with the any vertex of $N_{G'}(u)\cup \{u\}$, where $u$ is a 5-vertex in $G'$ since $G'$ contains a pair of adjacent neighbors of $u$.
Choose $(d_1-5)/2$ such edges, say $u_iv_i$ for each $i\in \{1,  \ldots, (d_1-5)/2)\}$. Denote by the graph $G$ from $G'$ by
deleting edges $u_iv_i$ and adding edge $uu_i$, $uv_i$ for each $i\in \{1, \ldots, (d_1-5)/2\}$.
It follows by Lemma~\ref{le22} that $G$ is a $Z_3$-connected realization of  degree sequence $(d_1, 4^{n-6}, 3^5)$, a contradiction. We complete our proof.

\end{document}